\documentclass{amsart}
\usepackage{graphicx}
\usepackage{amsfonts,amsmath,amstext,amsthm,amssymb}
\usepackage{hyperref}
\newtheorem*{theorem}{Theorem}

\newtheorem*{lemma}{Lemma}

\newtheorem*{remark}{Remark}
\newtheorem*{corollary}{Corollary}

\def\essup{{\rm essup\,}}

\def\supp{{\rm supp\,}}

\def\ZI{\ensuremath{\mathbb I}}
\def\ZN{\ensuremath{\mathbb N}}
\def\zI{\ensuremath{\mathcal I}}

\def\ZT{\ensuremath{\mathbb T}}

\def\md#1#2\emd{\ifx0#1
\begin{equation*} #2 \end{equation*}\fi  
\ifx1#1\begin{equation}#2\end{equation}\fi   
\ifx2#1\begin{align*}#2\end{align*}\fi   
\ifx3#1\begin{align}#2\end{align}\fi    
\ifx4#1\begin{gather*}#2\end{gather*}\fi  
\ifx5#1\begin{gather}#2\end{gather}\fi   
\ifx6#1\begin{multline*}#2\end{multline*}\fi  
\ifx7#1\begin{multline}#2\end{multline}\fi  
\ifx8#1\begin{multline*}\begin{split}#2\end{split}\end{multline*}\fi
\ifx9#1\begin{multline}\begin{split}#2\end{split}\end{multline}\fi
}
\newcommand {\e }[1]{(\ref{#1})}

\begin{document}

\title{Divergence of general localized operators on the sets of measure zero }%
\author{G. A. Karagulyan}%
\address{Institute of Mathematics of Armenian
National Academy of Sciences, Baghramian Ave.- 24b, 375019,
Yerevan, Armenia}%
\email{g.karagulyan@yahoo.com}%


\begin{abstract}
We consider sequences of linear operators $U_nf(x)$ with localization property. It is proved that for any set $E$ of measure zero there exists a set $G$ for which $U_n\ZI_G(x)$ diverges at each point $x\in E$. This result is a generalization of analogous theorems known for the Fourier sums operators with respect to different orthogonal systems.
\end{abstract}
\maketitle
\section{Introduction}
In 1876 P. Du Bois-Reymond \cite{Du} constructed an example of continuous function which trigonometric Fourier series diverges at some point. In 1923 A. N. Kolmogorov \cite{Kol} proved that for a function from $L^1(\ZT)$ divergence of Fourier series can be everywhere. On the other hand according to Carleson-Hant theorem (\cite{Car}, \cite{Hunt}) Fourier series of the functions from $L^p(\ZT),\, p>1,$ converge a.e.. A natural question is whether the Fourier series of a function from $L^p\,(p>1)$ or $C$ may diverge on a given arbitrary set of measure zero. In fact the investigation of this problem began before Carleson's theorem. First S. B. Stechkin \cite{Ste} in 1951 proved that for any set $E\subset \ZT$ of measure zero there exists a function $f\in L^2(0,2\pi)$ which Fourier series diverges on $E$. Then in 1963 L.~V.~Taikov \cite{Tai} proved $f$ can be taken from $L^p(0,2\pi)$  for any $1\le p<\infty $. In 1965 Kahane and Katznelson \cite{KaKa} proved the existence of a continuous complex valued function which diverges on a given set of measure zero. Essentially developing Kahane-Katznelson approach V.~V.~Buzdalin \cite{Buz} proved that for any set of measure zero there exists a continuous real function which Fourier series diverges on that set. The same question is investigated also for the other classical orthonormal systems.  Sh. V. Kheladze in \cite{Khe1} constructed a function from $L^p(0,1)$($1<p<\infty$) which Fourier-Walsh series diverges on a given set of measure zero. In another paper \cite{Khe2} he proved the same also for the Vilenkin systems. Then V. M. Bugadze \cite{Bug1} proved that for the Walsh system function in the mentioned theorem can be taken from $L^\infty$. In fact Bugadze proves the same also for Haar (\cite{Bug2}), Walsh-Paley and Walsh-Kachmaz systems(\cite{Bug1}). Haar system in such problems is considered also in the papers   M.~A.~Lunina \cite{Lun} and V.~I.~Prochorenko  \cite{Pro}. Recently  U.~Goginava \cite{Gog} proved that for any set of measure zero there exists a bounded function which Walsh-Fejer means diverges on that set. About other problems concerning the divergent Fourier series reader can find in the papers P.~L.~Ul'yanov \cite{Uly} and W.~L.~Wade \cite{Wade}.  We have noticed this phenomena is common for general sequences of linear operators with localization property. We consider sequences of linear operators
\md1\label{1}
U_nf(x)=\int_a^bK_n(x,t)f(t)dt,\quad n=1,2,\ldots,
\emd
with
\md1\label{2}
|K_n(x,t)|\le M_n.
\emd
We say the sequence \e{1} has a localization property ($L$-property) if for any function $f\in L^1(a,b) $ with $f(x)=1$ as $x\in I=(\alpha,\beta)(\subset [a,b])$ we have
\md0\label{3}
\lim_{n\to\infty }U_nf(x)=1\quad \hbox { as } x\in I,
\emd
and the convergence is uniformly in each closed set $A\subset I$. We prove the following
\begin{theorem}
If the sequence of operators \e{1} has a localization property, then for any set of measure zero $E\subset [a,b]$ there exists a set $G\subset [a,b]$ such that
\md0\label{4}
\liminf_{n\to\infty }U_n\ZI_G(x)\le 0,\quad \limsup_{n\to\infty }U_n\ZI_G(x)\ge 1 \hbox { for any } x\in E,
\emd
where $\ZI_G(x)$ denotes the characteristic function of $G$.
\end{theorem}
The result of the theorem can be applied to the Fourier partial sums operators with respect to all classical orthogonal systems(trigonometric, Walsh, Haar, Franklin and Vilenkin systems). Moreover instead of partial sum we can discuss also linear means of partial sums corresponding to an arbitrary regular method $T=\{a_{ij}\}$. All these operators have localization property. So the following corollary is an immediate consequence of the main result.
\begin{corollary}
Let $\Phi=\{\phi_n(x),\, n\in \ZN\}$, $x\in [a,b]$, be one of the above mentioned orthogonal systems and $T$ is an arbitrary regular linear method.  Then for any set $E$ of measure zero there exists a set $G\subset [a,b]$ such that the Fourier series of its characteristic function $f(x)=\ZI_G(x)$ with respect to $\Phi $ diverges on $E$ by $T$-method.
\end{corollary}
\begin{remark}
The function $f(x)$ in the corollary can not be continuous in general. There are variety of sequences of Fourier operators which uniformly converge while $f(x)$ is continuous.
\end{remark}
\noindent
The following lemma gives a bound for the kernels of operators \e{1} if $U_n$ has $L$-property.
\begin{lemma}
If the sequence of operators $U_n$ has $L$-property, then there exists a positive decreasing function $\phi(u)$, $u\in (0,+\infty)$, such that if $x\in [a,b]$ and $n\in\ZN $ then
\md1\label{5}
|K_n(x,t)|\le \phi(|x-t|) \hbox { for almost all } t\in [a,b].
\emd
\end{lemma}
\begin{proof}
We define
\md0
\phi(u)=\sup_{n\in \ZN,x\in[a,b]}\essup_{t:|t-x|\ge u}|K_n(x,t)|,
\emd
where $\essup_{t\in A}|g(t)|$ denotes $\|g\|_{L^\infty(A)}$.
It is clear $\phi(u)$ is decreasing and satisfies \e{5}, provided $\phi(u)<\infty $, $u>0$.
To prove $\phi(u)$ is finite, let us suppose the converse, that is $\phi(u_0)=\infty$ for some $u_0>0$. It means for any $\gamma >0$ there exist $l_\gamma\in \ZN$ and $c_\gamma\in [a,b]$ such that
\md1\label{92}
|K_{l_\gamma}(c_\gamma,t)|>\gamma,\quad t\in E_\gamma\subset [a,b]\setminus (c_\gamma-u_0,c_\gamma+u_0),\quad |E_\gamma |>0.
\emd
Consider the sequences $c_k$ and $l_k$ corresponding to the numbers $\gamma_k=k$, $k=1,2,\ldots $. We can fix an interval $I$ with $|I|=u_0/3$, where the sequence $\{c_k\}$ has infinitely many terms. Using this it can be supposed $c_\gamma\in I$ in \e{92} and therefore we will have $2I\subset (c_\gamma-u_0,c_\gamma+u_0)$ . So we can write
\md1\label{93}
c_\gamma\in I,\quad E_\gamma\subset [a,b]\setminus 2I.
\emd
Thus we chose a sequence $\gamma_k\nearrow \infty $ such that for corresponding sequences $m_k=l_{\gamma_k}$, $x_k=c_{\gamma_k}$ and $E_k=E_{\gamma_k}$ we have
\md3
x_k\subset I,\quad E_k\subset (a,b)\setminus 2I,\label{01}\\
|K_{m_k}(x_k,t)|\ge k^3,\quad t\in E_k,\label{9}\\
\sup_{1\le i<k}|U_{m_k}\ZI_{E_i}(x)|<1,\quad x\in I,\label{10}\\
|E_k|\cdot \max_{1\le i<k}M_{m_i}<1,(k>1).\label{11}
\emd
We do it by induction. Taking $\gamma_1=1$ we will get $m_1$ satisfying \e{9}. This follows from \e{92}. Now suppose we have already chosen the numbers $\gamma_k$ and $m_k$ satisfying \e{01}-\e{11} for $k=1,2,\ldots ,p$. According to $L$-property
$U_n\ZI_{E_i}(x)$ converge to $0$ uniformly in $I$ for any $i=1,2,\ldots ,p$. On the other hand because of \e{2} and \e{92} from $\gamma\to\infty $ follows $l_\gamma\to\infty $. Hence we can chose a number $\gamma_{p+1}>(p+1)^3$ such that the corresponding $m_{p+1}$ satisfies the inequality
\md1\label{14}
|U_{m_{p+1}}\ZI_{E_i}(x)|<1,\, x\in I,\, i=1,2,\ldots ,p.
\emd
This gives \e{10} in the case $k=p+1$. According to \e{92} and the bound $\gamma_{p+1}>(p+1)^3$ we will have also \e{9}. Finally, since $E_{p+1}$ may have enough small measure we can guarantee \e{11} for $k=p+1$. So the construction of the sequence $\gamma_k$ with \e{01}-\e{11} is complete. Now consider the function
\md1\label{15}
g(x)=\sum_{i=1}^\infty \frac{\ZI_{E_i}(x)}{k^2}.
\emd
We have $g\in L^1$ and $\supp g\subset [a,b]\setminus (2I)$.
Since $x_k\in I$, using the relations  \e{01}-\e{11}, we obtain
\md8
|U_{m_k}g(x_k)|\ge \\
&\ge\frac{|U_{m_k}\ZI_{E_k}(x_k)|}{k^2}-\sum_{i=1}^{k-1}\frac{|U_{m_k}\ZI_{E_i}(x_k)|}{i^2}-
\sum_{i=k+1}^\infty \frac{|U_{m_k}\ZI_{E_i}(x_k)|}{i^2}\\
&\ge k-\sum_{i=1}^{k-1}\frac{1}{i^2}-M_{m_k}\sum_{i=k+1}^\infty \frac{|E_i|}{i^2}\ge k-2.
\emd
This is a contradiction, because the convergence $U_ng(x)\to 0$ is uniformly on $I$ according to $L$-property.
\end{proof}

 We say a family $\zI$ of mutually disjoint semi-open intervals is a regular partition for an open set $G\subset (a,b)$ if $G=\cup_{I\in\zI}I$ and each interval $I\in \zI$ has two adjacent intervals $I^+,I^-\in \zI$ with
\md1\label{23}
2I\subset I^*=I\cup I^+\cup I^-.
\emd
 It is clear any open set has a regular partition.
 \begin{proof}[Proof of Theorem]
 For a given set $E$ of measure zero we will construct a definite sequence of open sets $G_k,k=1,2,\ldots,$ with regular partitions $\zI_k,\, k=1,2,\ldots $. They will satisfy the conditions

  1) if $I\in\zI_k$ and $I=[\alpha,\beta )$ then $\alpha,\beta\not\in E$,

  2) if $I,J\in \cup_{j=1}^k\zI_j$ then $J\cap I\in \{\varnothing,\, I,\, J\}$,

  3) $E\subset G_k\subset G_{k-1},(G_0=[a,b]).$

\noindent
In addition for any interval $I\in \zI $ we fix a number $\nu(I)\in \ZN$ such that

 4) if $I,J\in \cup_{j=1}^k\zI_j$ and $I\subset J$ then $\nu(I)\ge \nu(J)$,

 5) $\sup_{x\in I}|U_{\nu(I)}\ZI_{G_l}(x)-1|<1/k^2$, if $I\in \zI_k$ and $l\le k$,

 6) $\sup_{x\in I}|U_{\nu(I)}\ZI_{G_{k}}(x)|<1/k^2$, if $I\in \zI_l$ and $l<k$.

 \noindent
 We define $G_1$ and its partition $\zI_1$ arbitrarily, just ensuring the condition 1). It may be done because $|E|=0$ and so $E^c$ is everywhere dense on $[a,b]$. Then using $L$-property for any interval $I\in\zI_1$ we can find a number $\nu(I)\in \ZN$ satisfying 5) for $k=1$. Now suppose we have already chosen $G_k$ and $\zI_k$ with the conditions 1)-6) for all $k\le p$. Obviously we can chose an open set $G_{p+1}$, $E\subset G_{p+1}\subset G_p$, satisfying 1), 2) and the bound
 \md0
 |G_{p+1}\cap I|<\delta(I),\quad I\in \cup_{k=1}^p\zI_k,
 \emd
 where
 \md0
 \delta(I)=\frac{1}{6(p+1)^2\max\left\{M_{\nu(I)},M_{\nu(I^+)},M_{\nu(I^-)},\frac{\phi(|I|/2)}{|I|}\right\}},
 \emd
and the function $\phi(u)$ is taken from the lemma. Suppose $I\in \zI_l$ and $l<p+1$. We have
\md1\label{26}
|U_{\nu(I)}\ZI_{G_{p+1}}(x)|\le |U_{\nu(I)}\ZI_{G_{p+1}\cap I^*}(x)|+|U_{\nu(I)}\ZI_{G_{p+1}\cap (I^*)^c}(x)|.
\emd
Using the lemma and the bound
\md0
\delta(J)\le \frac{|J|}{6\phi(|J|/2)(p+1)^2},\quad J\in \zI_l,
\emd
for any $x\in I$ we get
\md9\label{24}
|U_{\nu(I)}\ZI_{G_{p+1}\cap (I^*)^c}(x)|\\
&\le \sum_{J\in\zI_l:J\neq I,I^+,I^-}\int_{G_{p+1}\cap J}\phi(|x-t|)dt\\
&\le \sum_{J\in\zI_l:J\neq I,I^+,I^-}\int_{G_{p+1}\cap J}\phi\left(\frac{|J|}{2}\right)dt\\ &\le \sum_{J\in\zI_l:J\neq I,I^+,I^-}|G_{p+1}\cap J|\phi\left(\frac{|J|}{2}\right)\\
&\le \sum_{J\in\zI_l:J\neq I,I^+,I^-}\delta(J)\phi\left(\frac{|J|}{2}\right)\\
&\le \frac{1}{6(p+1)^2}\sum_{J\in\zI_l:J\neq I,I^+,I^-}|J|
<\frac{1}{6(p+1)^2},\quad x\in I.
\emd
On the other hand we have
\md0
\delta(I),\delta(I^+),\delta(I^-)\le \frac{1}{6\cdot (p+1)^2M_{\nu(I)}},
\emd
and therefore
\md9\label{25}
|U_{\nu(I)}\ZI_{G_{p+1}\cap I^*}(x)|&\le M_{\nu(I)}|G_{p+1}\cap I^*|\\
&\le M_{\nu(I)}(\delta(I)+\delta(I^+)+\delta(I^-))
\le \frac{1}{2(p+1)^2},\quad x\in [a,b].
\emd
Combining \e{26},\e{24} and \e{25} we get 6) in the case $k=p+1$. Now we chose the partition $\zI_{p+1}$ satisfying just conditions 1) and 2). Using $L$ property we may define numbers $\nu(I)$ for $I\in \zI_{p+1}$ satisfying the condition 5) with $k=p+1$. Hence the construction of the sets $G_k$ is complete.
Now denote
\md0
G=\bigcup_{i=1}^\infty (G_{2i-1}\setminus G_{2i}),
\emd
we have
\md0
U_n\ZI_G(x)=\sum_{k=1}^\infty (-1)^{k+1}U_n\ZI_{G_k}(x).
\emd
For any $x\in E$ there exists a unique sequence $I_1\supset I_2\supset\ldots\supset I_k\supset \ldots $, $I_k\in \zI_k$, such that $x\in I_k$, $k=1,2,\ldots $. According to 6) for $l> k$ we have
\md0
|U_{\nu(I_k)}\ZI_{G_l}(x)|\le \frac{1}{l^2},\quad l>k.
\emd
From 5) it follows that
\md0
|U_{\nu(I_k)}\ZI_{G_l}(x)-1|\le \frac{1}{k^2},\quad l\le k.
\emd
Thus we obtain
\md6
|U_{\nu(I_k)}\ZI_G(x)-\sum_{l=1}^k(-1)^{l+1}|\\
\le \sum_{l=1}^k|U_{\nu(I_k)}\ZI_{G_l}(x)-1|+\sum_{l=k+1}^\infty |U_{\nu(I_k)}\ZI_{G_l}(x)|\\
\le k\cdot \frac{1}{k^2}+\sum_{l=k+1}^\infty \frac{1}{l^2}<\frac{2}{k}
\emd
Since the sum $\sum_{i=1}^k(-1)^{k+1}$ takes values $0$ and $1$ alternately we get
\md0
\lim_{t\to \infty} U_{\nu(I_{2t})}\ZI_G(x)=0,\quad
\lim_{t\to \infty} U_{\nu(I_{2t+1})}\ZI_G(x)=1
\emd
for any $x\in E$. The proof of theorem is complate.
\end{proof}

\end{document}